\documentclass[12pt]{amsart}
\usepackage{amssymb}
\usepackage[mathscr]{eucal}
\usepackage{epsf}
\usepackage{epsfig}
\usepackage{color}
\DeclareGraphicsExtensions{.pstex,.eps,.epsi,.ps}

 \newtheorem{thm}{Theorem}[section]
 \newtheorem{cor}[thm]{Corollary}
 \newtheorem{lem}[thm]{Lemma}
 \newtheorem{prop}[thm]{Proposition}
 \theoremstyle{definition}
 
 \theoremstyle{remark}
 
 \numberwithin{equation}{section}

 \newcommand{\Real}{\mathbb{R}}

 \newcommand{\hor}{\text{hor}}
 
 \newcommand{\tr}{\textbf{tr}}
 
 \newcommand{\vol}{\textbf{vol}}

 \newcommand{\Rm}{\textbf{Rm}}

\newcommand{\e}{\epsilon}

\newcommand{\cJ}{\mathcal J}

\begin{document}

\title[Ricci on Sasakian]{Ricci curvature lower bounds on Sasakian manifolds}

\author{Paul W.Y. Lee}
\email{wylee@math.cuhk.edu.hk}
\address{Room 216, Lady Shaw Building, The Chinese University of Hong Kong, Shatin, Hong Kong}

\date{\today}

\maketitle

\begin{abstract}
Measure contraction property is a synthetic Ricci curvature lower bound for metric measure spaces. 
We consider Sasakian manifolds with non-negative Tanaka-Webster Ricci curvature equipped with the metric measure space structure defined by the sub-Riemannian metric and the Popp measure. We show that these spaces satisfy the measure contraction property $MCP(0,N)$ for some positive integer $N$. We also show that the same result holds when the Sasakian manifold is equipped with a family of Riemannian metrics extending the sub-Riemannian one. 
\end{abstract}


\section{Introduction}

In the last decade, there is a surge of interest in the study of metric measure spaces satisfying synthetic Ricci curvature lower bounds. This area of research begins with the work by Lott-Villani \cite{LoVi1} and Sturm \cite{St1, St2}. In there, they introduce a notion of synthetic Ricci curvature lower bound for length spaces equipped with a measure, called curvature-dimension condition, using the theory of optimal transportation. Length spaces are metric spaces for which the distance between any two points is the same as the length of curves, called geodesics, connecting the two points. Since then, numerous work is devoted to the search of other notions of synthetic Ricci curvature lower bound and extending various well-known results on Riemannian manifolds with Ricci curvature lower bound to these spaces. Here is a non-exhaustive list of related works \cite{LoVi1, St1, St2, Oh1, Ol, AmGiSa, Gi} (see also \cite{Vi} for an introduction and a more complete list of references). 

In this paper, we focus on another notion of synthetic Ricci curvature lower bound called the measure contraction property $MCP(k, n)$ introduced and studied in \cite{St1, St2, Oh1}. Roughly speaking, a length space equipped with a measure satisfies the measure contraction property if for each Borel set and a point in the length space, there is an interpolation by geodesics between them such that the ratio of the measure of this interpolation to that of the set can be estimated from below by the corresponding ratio in a space form. When the metric measure space is a Riemannian manifold of dimension $n$ equipped with the Riemannian volume, then the condition $MCP(k,n)$ is equivalent to the Ricci curvature bounded of the Riemannian manifold being bounded below by $k$. 

The situation is quite different in the sub-Riemannian setting. In this case, the curvature-dimension condition does not hold. In \cite{Ju}, it was proved that the Heisenberg group of dimension $2n+1$ satisfies $MCP(0, 2n+3)$. This was extended to Sasakian manifolds equipped with the sub-Riemannian metric and the Popp volume in \cite{AgLe, LeLiZe} under curvature conditions defined by the Tanaka-Webster curvature (see section \ref{Sa} for the precise definitions of Sasakian manifolds and the Tanaka-Webster curvature). In this paper, we relax this condition to the non-negativity of the Tanaka-Webster Ricci curvature. 

\begin{thm}\label{main}
Let $M$ be a manifold equipped with a Sasakian structure with non-negative Tanaka-Webster Ricci curvature. Then the metric measure space $(M, d_S, \vol)$ satisfies $MCP(0,N)$ for some positive integer $N$, where $d_S$ is the sub-Riemannian metric and $\vol$ is the Popp measure. 
\end{thm}

There are numerous consequences of Theorem \ref{main}. It follows immediately from the condition $MCP(0,N)$ that the metric measure space $(M, d_S, \vol)$ satisfies the doubling condition. This means that if the $B_x(r)$ is a ball of radius $r$ centered at $x$, then $\frac{\vol(B_x(2r))}{\vol(B_x(r))}$ is bounded above by a constant independent of the point $x$ and and the radius $r$.  It also follows from \cite{LoVi2} (see also \cite{LeLiZe}) and a classical argument  due to Jerison \cite{Je} that the $L^p$-Poinca\'re inequalities hold. The doubling condition and the $L^2$-Poinca\'re inequality together also gives parabolic Harnack inequality \cite{Mo1, Mo2, Gr, Sa1} and heat kernel bounds \cite{FoSt} for the sub-elliptic heat equation  (see also \cite{KuSt} for the converse). 

It also follows from the doubling condition, the $L^2$-Poinca\'re inequality, and the work in \cite{CoMi} that the linear space of harmonic functions with polynomial growth of a fixed degree is finite dimensional. Finally, our approach  can be used to prove a Bonnet-Myers type theorem (see Theorem \ref{Myers}). 
This covers and gives an alternative approach to all the results in \cite{BaBoGa} and \cite{BaGa} at least in the Sasakian case. In fact, it follows from the proof of Theorem \ref{main} that a Laplacian comparison theorem holds (see Corollary \ref{LapCom}). This can be used to prove a local differential Harnack inequality for the sub-elliptic heat equation which improves the global result in \cite{BaGa}. 

One can also combine the doubing condition, the $L^p$-Poinca\'re inequality, and the work in \cite{CoHoSa} to obtain the Harnack inequality and the Liouville theorem for solutions of the sub-elliptic version of the $p$-Laplace equation. It also gives rigidity results for quasi-regular mappings. These consequences are, to the knowledge of the author, not covered by any previous work. 

In addition, we also prove the measure contraction property holds for a family of Riemannian manifolds for which the Riemannian metrics $\left<\cdot,\cdot\right>^\e$ extend the sub-Riemannian one (see Section \ref{Riem} for the precise definition). Measure contraction property for such a family was considered in \cite{Ri} for the three dimensional Heisenberg group and then in \cite{Le2} for the Sasakian case under a curvature condition defined by the Tanaka-Webster curvature. We relax this condition to the non-negativity of the Tanaka-Webster Ricci curvature in the following theorem. Consequences analogous to those mentioned above for Theorem \ref{main} also hold in these Riemannian cases. 

\begin{thm}\label{main2}
Let $M$ be a manifold equipped with a Sasakian structure with non-negative Tanaka-Webster Ricci curvature. Then the family of metric measure spaces $(M, d_\e, \vol)$ satisfies $MCP(0,N)$ for some positive integer $N$, where $d_\e$ is the Riemannian distance defined by the Riemannian metric $\left<\cdot,\cdot\right>^\e$. 
\end{thm}

In section \ref{Sub}, we will give precise definition of the measure contraction property in the sub-Riemannian setting. We will also begin the proof of Theorem \ref{main} in this section. In section \ref{Sa}, we discuss properties of Sasakian manifolds and the corresponding Tanaka-Webster curvature. The complete proof of Theorem \ref{main} is also contained in this section. Section \ref{Riem} is devoted to the proof of Theorem \ref{main2}. The proof of the Bonnet-Myers type theorem mentioned above is contained in the last section. 

\smallskip

\section*{Acknowledgements}
The question answered in this paper is raised by Professor Fabrice Baudoin during one of his talks. The author would like to thank him for his encouragement. This research was supported by the Research Grant Council of Hong Kong (RGC Ref. No. CUHK404512). 

\smallskip

\section{Sub-Riemannian Manifolds and the Measure Contraction Property}\label{Sub}

In this section, we recall various notions and facts about contact sub-Riemannian manifolds and the measure contraction property. 

A sub-Riemannian manifold is a manifold $M$ equipped with a distribution $\mathcal D$ and a positive definite metric $\left<\cdot, \cdot\right>_{SR}$ defined on $\mathcal D$. A smooth path in $M$ is a horizontal path if it is everywhere tangent to the distribution $\mathcal D$. Assume that $\mathcal D$ is bracket generating meaning that there is a positive integer $k$ such that the iterated Lie-brackets of vector fields in $\mathcal D$ up to order $k$ generate the each tangent space of $M$. Under this assumption, the Chow-Rashevskii theorem (see \cite{Mo}) guarantees that any two points can be connected by a horizontal path. Because of this, one can define a distance $d_S$, called the sub-Riemannian or Carnot-Caratheodory distance. The distance $d_S(x,y)$ between two points $x$ and $y$ is given by the length of the shortest horizontal path connecting $x$ and $y$. Here the length is measured by the given sub-Riemannian metric $\left<\cdot, \cdot\right>_{SR}$. 

In this paper, the distributions of the sub-Riemannian manifolds are contact. A distribution $\mathcal D$ is contact if there is a 1-form $\eta$ such that $\ker\eta=\mathcal D$ and the restriction of $d\eta$ to $\mathcal D$ is non-degenerate. Let $X_1$ and $X_2$ be two vector field in $\mathcal D$. It follows that $d\eta(X_1, X_2)=-\eta([X_1,X_2])$. Therefore, by the non-degeneracy condition, $\mathcal D$ is bracket generating. 

Let $V$ be the Reeb field defined by $\eta(V)=1$ and $d\eta(V,\cdot)=0$. We also equipped $M$ with a Riemannian metric $\left<\cdot,\cdot\right>$ such that $V$ is of length 1, $V$ is orthogonal to $\ker\eta$, and $\left<\cdot,\cdot\right>$ extends the sub-Riemannian one $\left<\cdot,\cdot\right>_{SR}$. Let $\vol$ be the Riemannian volume form of this Riemannian metric. This volume form coincides with the Popp volume in the sub-Riemannian case (see \cite{Mo} for the definition of Popp measure for general sub-Riemannian manifolds). We will also denote the corresponding measure with the same symbol. For the rest of this paper, we consider the following metric measure space $(M, d_S, \vol)$. 

Let $x_0$ be a point on the manifold $M$ and let $f_0(x)=-\frac{1}{2}d^2_S(x_0,x)$. By \cite[Theorem 1]{CaRi} (see also \cite{FiRi}), the function $f_0$ is locally semiconcave on $M-\{x_0\}$. Therefore, by Alexandrov theorem (see \cite{EvGa, Vi}), $f_0$ is twice differentiable $\vol$ almost everywhere. 

Let $H:T^*M\to\Real$ be the sub-Riemannian Hamiltonian defined by $H(x,p)=\frac{1}{2}|p^b|^2_{SR}$, where $p^b$ is the unique tangent vector in the distribution $\ker(\eta)$ satisfying $\left<p^b,v\right>_{SR}=p(v)$ for all $v$ in $\ker(\eta)$. Let $e^{t\vec H}$ be the Hamiltonian flow of $H$ and let $\varphi_t:M\to M$ be the map defined by 
\[
\varphi_t(x)=\pi(e^{t\vec H}(x,df_0)), 
\]
where $\pi:T^*M\to M$ is the projection map. It follows as in \cite{Mc, EvGa} that the map $x\mapsto (x, df_0)$ and hence $\varphi_t$ for all $t$ in $[0,1]$ are Borel maps. 

The measure contraction property $MCP(0,N)$ for the metric measure space $(M, d_S, \vol)$ can now be rewritten in terms of the map $\varphi_t$ as
\begin{equation}\label{MCP}
\vol(\varphi_t(U))\geq (1-t)^N\vol(U)
\end{equation}
for all Borel set $U$. Here $\varphi_t(U)$ is the interpolation between the set $U$ and the point $x_0$ by geodesics mentioned in the introduction. 

Let $Z$ be the set of all points where $f_0$ is twice differentiable. For each point $x$ in $Z$, the curve $t\mapsto \varphi_t(x)$ is the unique length minimizing sub-Riemannian geodesic starting from $x$ and ending at $x_0$. It follows that $\varphi_t$ is injective on $Z$ for all $0\leq t<1$. In particular, in order to show $MCP(0,N)$, it is enough to show that 
\begin{equation}\label{MCPII}
\vol(U)\geq (1-t)^N(\varphi_t)_*\vol (U)
\end{equation}
for any Borel set $U$. 

Finally, we are going to show that 
\begin{equation}\label{Jac}
\det(d\varphi_t)=e^{\int_0^t\Delta_Hf_s(\varphi_s)ds}
\end{equation}
on $Z$, where $f_t(x)=-\frac{1}{2(1-t)}d^2_S(x_0,x)$ and $\Delta_H$ denotes the sub-Laplacian. The sub-Laplacian is defined by 
\[
\Delta_H f(x)=\sum_{i=1}^{2n}\left<\nabla^2 f(v_i),v_i\right>,
\]
where $\{v_1,...,v_{2n}\}$ is an orthonormal basis of $\ker(\eta_x)$ and $\nabla^2$ is the Hessian with respect to the Riemannian metric defined above. 

Therefore, it follows from \cite[Theorem 11.1]{Vi} that the measure $(\varphi_t)_*\vol$ is concentrated on $\varphi_t(Z)$, it is absolutely continuous with respect to $\vol$, and its density $g$ with respect to $\vol$ satisfies $g(\varphi_t)\det(d\varphi_t)=1$. 

On the other hand, if we can show 
\begin{equation}\label{Lap}
\Delta_H f_t(\varphi_t(x))\geq -\frac{N}{1-t}
\end{equation}
for all $x$ in $Z$, then it follows from (\ref{Jac}) and (\ref{Lap}) that 
\[
\det(d\varphi_t)\geq e^{-N\int_0^t\frac{1}{1-s}ds}= (1-t)^N
\]
and so $g\geq \frac{1}{(1-t)^N}$ for $\vol$ a.e. Hence, (\ref{MCPII}) holds. 

The rest of this section is devoted to the proof of (\ref{Jac}). (\ref{Lap}) will be shown in other sections under additional assumptions stated in Theorem \ref{main}. 

First, note that if the curve $t\mapsto \pi(e^{t\vec H}(x_0,p_0))$ is a length minimizing geodesic, then, for each $t$ in $(0, 1)$, $p\mapsto \pi(e^{t\vec H}(x_0,p))$ is a diffeomorphism from a neighbourhood of $p_0$ in the cotangent space $T^*_xM$ to a neighbourhood of $\pi(e^{t\vec H}(x_0,p_0))$ in the manifold $M$. It follows from this that the function $f_t$ is smooth at $\varphi_t(x)$ for each $x$ and each $0<t<1$, so the term $\Delta_H f_t(\varphi_t)$ in (\ref{Jac}) makes sense. 

Let $x$ be in the set $Z$ where $f_t$ is differentiable. Then $\nabla_H f_0$ is the initial velocity of the geodesic $t\mapsto \varphi_t(x)$, where $\nabla_H f_0$ is the horizontal gradient of the function $f_0$ defined by projecting $\nabla f_0$ orthogonally onto the distribution $\ker\eta$. It follows that $|\nabla_H f_0|=d_{S}(x_0 ,x)$ and so $f_t$ satisfies the Hamilton-Jacobi equation 
\begin{equation}\label{HJ}
\dot f_t=-\frac{1}{2}|\nabla_H f_t|^2=-\frac{1}{2}|\nabla f_t|^2+\frac{1}{2}\left<\nabla f_t,V\right>^2. 
\end{equation}
Therefore, by the method of characteristics, 
\begin{equation}\label{flow}
\dot\varphi_t(x)=\nabla_H f_t(\varphi_t(x))=\nabla f_t(\varphi_t(x))-\left<\nabla f_t,V\right>_{\varphi_t(x)}V(\varphi_t(x)).
\end{equation}

Let $v_0(t)=V(\varphi_t(x))$ and let $v_1(t),...,v_{2n}(t)$ be a family of orthonormal frames of $\ker\eta$ defined along the path $t\mapsto\varphi_t(x)$. This family can be chosen in such a way that $\dot v_i(t)$ is contained in the orthogonal complement of $\ker\eta$. 

Let $A(t)$ and $W(t)$ be the matrices defined by 
\[
d\varphi_t(v_i(0))=\sum_{j=0}^{2n}A_{ij}(t)v_j(t)
\]
and 
\[
\dot v_i(t)=\sum_{j=0}^{2n}W_{ij}(t)v_j(t). 
\]
Note that $W(t)$ is skew symmetric. 

By differentiating (\ref{flow}) in $x$, we obtain 
\[
\begin{split}
&\sum_{j=0}^{2n}\dot A_{ij}(t)v_j(t)+\sum_{j, k =0}^{2n}A_{ik}(t)W_{kj}(t)v_j(t)=\frac{d}{dt} d\varphi_t(v_i(0))\\
&=\sum_{j=0}^{2n}A_{ij}(t)\Big[\nabla^2 f_t(v_j(t))-\left<\nabla^2 f_t(v_j(t)), v_0(t)\right>v_0(t)\\
&-\left<\nabla f_t,\nabla V(v_j(t))\right>v_0(t)-\left<\nabla f_t,V\right>_{\varphi_t}\nabla V(v_j(t))\Big]. 
\end{split}
\]

Since $W(t)$ is skew symmetric, it follows that 
\[
\begin{split}
\tr(A(t)^{-1}\dot A(t))&=\Delta_H f_t(\varphi_t(x))-\left<\nabla f_t,\nabla V(V)\right>_{\varphi_t(x)}\\
&-\left<\nabla f_t,V\right>_{\varphi_t(x)}\tr(\nabla V(x)). 
\end{split}
\]

Therefore, by \cite[Theorem 4.5 and Lemma 6.2]{Bl}, 
\[
\tr(A(t)^{-1}\dot A(t))=\Delta_H f_t(\varphi_t(x)).
\]

Finally, since $A(0)=I$ and $\frac{d}{dt}\log\det(A(t))=\tr(A(t)^{-1}\dot A(t))$, (\ref{Jac}) follows. 

\smallskip

\section{Sasakian Manifolds and Measure Contraction Property}\label{Sa}

In this section, we recall the definition of a Sasakian manifold and prove some properties needed for this paper. We then give a proof of (\ref{Lap}) and hence Theorem \ref{main}. 

Let $\eta$ be a contact form on a manifold $M$ of dimension $2n+1$ and let $V$ be the Reeb field. Let $\cJ$ be a $(1,1)$-tensor satisfying $\cJ V=0$, $\cJ^2X=-X$ for all horizontal vector field $X$ (i.e. $X$ is  contained in the distribution $\ker\eta$), and
\begin{equation}\label{conmet}
d\eta(X_1,X_2)=\left<X_1,\cJ X_2\right>.
\end{equation}
We call the structure $(\cJ,V,\eta,\left<\cdot,\cdot\right>)$ is Sasakian if the following holds for all vector fields $Y_1$ and $Y_2$:
\begin{equation}\label{Nij}
\begin{split}
&d\eta(Y_1,Y_2)V\\
&=-\cJ^2[Y_1,Y_2]+\cJ[\cJ Y_1,Y_2]+\cJ[Y_1,\cJ Y_2]-[\cJ Y_1,\cJ Y_2].
\end{split}
\end{equation}

Next, we state some properties of the Levi-Civita connection which is needed for this paper. The proof can be found in \cite{Bl, Le2}. 

\begin{prop}\label{ind}
Assume that the structure $(\cJ,V,\eta,\left<\cdot,\cdot\right>)$  is Sasakian. Then
\begin{enumerate}
\item $\nabla_YV=-\frac{1}{2}\cJ Y$, 
\item $\nabla_{X_1}\cJ(X_2)=\frac{\left<X_1,X_2\right>}{2}V$, 
\item $\nabla_X\cJ(V)=-\frac{1}{2}X$, 
\item $\nabla_V\cJ=0$, 
\end{enumerate}
for all vector fields $Y$ and all horizontal vector fields $X_1$, $X_2$. 
\end{prop}

The Tanaka-Webster connection $\bar\nabla$ is defined by
\[
\bar\nabla_{Y_1}Y_2=\nabla_{Y_1}Y_2+\frac{\left<V,Y_2\right>}{2}\cJ Y_1-\frac{1}{2}\left<\cJ Y_1,Y_2\right>V+\frac{\left<V,Y_1\right>}{2}\cJ Y_2.
\]
The corresponding curvature is the Tanaka-Webster curvature and it is denoted by $\overline\Rm$. 

The Riemann curvature tensor $\Rm$ and the Tanaka-Webster curvature $\overline\Rm$ are related as follows.

\begin{prop}\label{Rm}
Assume that the structure $(J,V,\eta,\left<\cdot,\cdot\right>)$ is Sasakian. Then
\begin{enumerate}
\item $\Rm(Y_1,Y_2)V=\frac{\left<Y_2,V\right>}{4}(Y_1)_\hor-\frac{\left<Y_1,V\right>}{4}(Y_2)_\hor$,
\item $\overline{\Rm}(X_2,X_3)X_1=\Rm(X_2,X_3)X_1+\frac{\left<\cJ X_3,X_1\right>}{4}\cJ X_2\\ -\frac{\left<\cJ X_2,X_1\right>}{4}\cJ X_3-\frac{\left<\cJ X_2,X_3\right>}{2}\cJ X_1$,
\item $\overline\Rm(Y_1,Y_2)V=0$,
\item $\overline\Rm(X_1,V)X_2=0$,
\end{enumerate}
for all vector fields $Y_1$, $Y_2$ and all horizontal vector fields $X_1$, $X_2$, $X_3$.  
\end{prop}

\begin{proof}[Proof of Theorem \ref{main}]
We start off with the following simple lemma:

\begin{lem}\label{conserve}
Assume that the manifold is Sasakian. Then 
\begin{enumerate}
\item $\frac{d}{dt}\left<\nabla f_t,V\right>_{\varphi_t}=0$, 
\item $\frac{d}{dt}|\nabla_H f_t|^2_{\varphi_t}=0$, 
\item $\frac{D^2}{dt^2}\varphi_t=\left<\nabla f_t,V\right>_{\varphi_t}\mathcal J\nabla f_t(\varphi_t)$. 
\end{enumerate}
\end{lem}

\begin{proof}
By differentiating the equation (\ref{HJ}) and applying Proposition \ref{ind}, we obtain 
\[
\begin{split}
\left<\nabla\dot f_t, Y\right>&=-\left<\nabla^2 f_t(\nabla f_t),Y\right>+\left<\nabla f_t,V\right>(\left<\nabla^2 f_t(Y),V\right>+\left<\nabla f_t,\nabla_Y V\right>)\\
&=-\left<\nabla^2 f_t(\nabla_H f_t) ,Y\right>+\frac{1}{2}\left<\nabla f_t,V\right>\left<\mathcal J\nabla f_t, Y\right>.  
\end{split}
\]
It follows from this and (\ref{flow}) that $\frac{d}{dt}\left<\nabla f_t,V\right>_{\varphi_t}=0$. Therefore, by Proposition \ref{ind} again, 
\[
\begin{split}
\frac{D^2}{dt^2}\varphi_t(x)&=\frac{1}{2}\left<\nabla f_t,V\right>_{\varphi_t}\mathcal J\nabla f_t(\varphi_t)-\left<\nabla f_t,V\right>_{\varphi_t}\nabla V(\dot \varphi_t)\\
&=\left<\nabla f_t,V\right>_{\varphi_t}\mathcal J\nabla f_t(\varphi_t). 
\end{split}
\]
The second assertion follows from similar arguments .
\end{proof}

Next, we define a family of orthonormal frames similar to the one in the previous section. Similar orthonormal frames were first introduced by the author in \cite{Le3}. It was also used in \cite{Le2, Le4, Le5} in many other geometric situations. 

\begin{lem}\label{frame}
There is a family of orthonormal frame $\{v_0(t), ..., v_{2n}(t)\}$ defined along the path $t\mapsto\varphi_t(x)$ such that $v_0(t)=V(\varphi_t)$ and $\dot v(t)=W(t)v(t)$, where 
\[
W=\left(
       \begin{array}{cccc}
         0 & 0 & -\frac{|\nabla_Hf_0|}{2} & 0 \\
         0 & 0 & \left<\nabla f_0, V\right> & 0 \\
         \frac{|\nabla_Hf_0|}{2} & -\left<\nabla f_0, V\right> & 0 & 0 \\
         0 & 0 & 0 & O_{2n-2} \\
       \end{array}
     \right).
\]
\end{lem}

\begin{proof}
Let $v_0(t)=V(\varphi_t)$, $v_1(t)=\frac{1}{|\nabla_H f_0|}\nabla_Hf_t(\varphi_t)$, and $v_2(t)=\frac{1}{|\nabla_H f_0|}\mathcal J \nabla f_t(\varphi_t)$. 

It follows that from Proposition \ref{ind} that
\[
\begin{split}
\frac{D}{dt} v_0(t)&=\nabla V(\dot \varphi_t)=-\frac{1}{2}\mathcal J\nabla f_t(\varphi_t)=-\frac{|\nabla_H f_0|}{2}v_2(t). 
\end{split}
\]

It follows from Lemma \ref{conserve} that 
\[
\begin{split}
\frac{D}{dt} v_1(t)=\frac{1}{|\nabla_H f_0|}\frac{D^2}{dt^2}\varphi_t=\frac{\left<\nabla f_0,V\right>}{|\nabla_H f_0|}\mathcal J\nabla f_t(\varphi_t)=\left<\nabla f_0,V\right>v_2(t)
\end{split}
\]

It also follows from Proposition \ref{ind} that 
\[
\begin{split}
&\frac{D}{dt} v_2(t)=\frac{1}{|\nabla_H f_0|}\left(\nabla_{\dot\varphi_t}\mathcal J(\nabla f_t(\varphi_t))+\mathcal J\nabla\dot f_t(\varphi_t)+\mathcal J\nabla^2 f_t(\nabla_H f_t(\varphi_t))\right)\\
&=\frac{1}{|\nabla_H f_0|}\Big(\nabla_{\nabla_H f_t}\mathcal J(\nabla_H f_t(\varphi_t))\\
&+\left<\nabla f_t,V\right>_{\varphi_t}\nabla_{\nabla_H f_t}\mathcal J(V(\varphi_t))-\frac{1}{2}\left<\nabla f_t,V\right>_{\varphi_t}\nabla_H f_t(\varphi_t)\Big)\\
&=\frac{1}{2}|\nabla_Hf_0|V(\varphi_t)-\frac{\left<\nabla f_0,V\right>}{|\nabla_H f_0|}\nabla_H f_t(\varphi_t) \\
&=\frac{1}{2}|\nabla_Hf_0|v_0(t)-\left<\nabla f_0,V\right>v_1(t). 
\end{split}
\]
Finally, we can choose $v_3(t),...,v_{2n}(t)$ such that $\frac{D}{dt} v_i(t)$ is contained in the span of $v_0(t), v_1(t), v_2(t)$ for each $i=3,...,2n$. This can be achieved by first extending the frame $\{v_0(t), v_1(t), v_2(t)\}$ to an orthonormal  frame $\{v_0(t), v_1(t), v_2(t), w_3(t), ..., w_{2n}(t)\}$. Then let $v_i(t)=\sum_{j=3}^{2n}O_{ij}(t)w_j(t)$ such that $O(t)$ is an orthogonal matrix satisfying $\dot O(t)=-O(t)\Lambda(t)$, where $\Lambda_{ij}(t)=\left<\frac{D}{dt} w_i(t),w_j(t)\right>$. It follows that $\frac{D}{dt} v_i(t)$ is orthogonal to $v_k(t)$ for each $i, k=3,...,2n$. It also follows that 
\[
\left<\frac{D}{dt} v_i(t),v_j(t)\right>=-\left<v_i(t),\frac{D}{dt} v_j(t)\right>=0
\] 
for each $j=0,1,2$ and $i=3,...,2n$. Therefore, $\frac{D}{dt} v_i(t)=0$. It also follows from this and Proposition \ref{ind} that one can pick $v_{2i}(t)=\cJ v_{2i-1}(t)$. 
\end{proof}

\begin{lem}\label{Ricc}
Let $S(t)$ be the matrix defined by 
\[
S_{ij}(t)=\left<\nabla^2 f_t(v_i(t)),v_j(t)\right>. 
\]
Then $S(t)$ satisfies the following matrix Riccati equation 
\[
\begin{split}
&\dot S(t)-US(t)-S(t)U^T+\bar R(t)+E+S(t)DS(t)=0,
\end{split}
\]
with the condition $\lim_{t\to 1}S(t)^{-1}=0$, where $a=|\nabla_H f_0|$, $b=\left<\nabla f_0,V\right>$, $\bar R_{ij}(t)=\left<\bar Rm(v_i(t)),\nabla_H f_t)\nabla_H f_t,v_j(t)\right>$,  
\[
\begin{split}
&E=\left(\begin{array}{cc}
E_1 & O\\
O & \frac{b^2}{4}I_{2n-2}
\end{array}\right), 
U=\left(\begin{array}{cc}
U_1 & O\\
O & -\frac{b}{2}J_{2n-2}
\end{array}\right), D=\left(\begin{array}{cc}
0  & 0\\
0 & I_{2n}
\end{array}\right), \\
&E_1=\left(\begin{array}{ccc}
\frac{a^2}{4} & -\frac{ab}{4} & 0 \\
-\frac{ab}{4} & \frac{b^2}{4} & 0\\
0 & 0 & -a^2+\frac{b^2}{4}
\end{array}\right), U_1=\left(
       \begin{array}{cccc}
         0 & 0 & -\frac{a}{2} \\
         0 & 0 & \frac{b}{2} \\
        a & -\frac{b}{2} & 0
       \end{array}
     \right), 
\end{split}
\]
and $J_{2n-2}$ is the $(2n-2)\times (2n-2)$ matrix such that the $ij$-th entry is given by $\left<\cJ(v_{i+2}(t),v_{j+2}(t)\right>$
\end{lem}

\begin{proof}
By differentiating (\ref{HJ}), 
\[
\begin{split}
&\nabla^2\dot f_t(v,w)+\nabla^3 f_t(w, \nabla f_t,v)+\nabla^2 f_t(\nabla^2 f_t(w),v)\\
&-\left<\nabla^2 f_t(w),V\right>\nabla^2 f_t(V,v)-\left<\nabla f_t,\nabla V(w)\right>\nabla^2 f_t(V,v)\\
&-\left<\nabla f_t,V\right>\nabla^3 f_t(w, V,v)-\left<\nabla f_t,V\right>\nabla^2 f_t(\nabla V(w) ,v)\\
&-\frac{1}{2}\left<\nabla^2 f_t(w),V\right>\left<\mathcal J\nabla f_t,v\right>-\frac{1}{2}\left<\nabla f_t,\nabla V(w)\right>\left<\mathcal J\nabla f_t,v\right>\\
&-\frac{1}{2}\left<\nabla f_t,V\right>\left<(\nabla_w \mathcal J)\nabla f_t,v\right>-\frac{1}{2}\left<\nabla f_t,V\right>\left<\mathcal J\nabla^2 f_t(w),v\right>=0. 
\end{split}
\]

By the Ricci identity and Proposition \ref{ind}, 
\[
\begin{split}
&\nabla^2\dot f_t(v,w)+\nabla^3 f_t(\nabla_H f_t,w,v)-\left<\Rm(w,\nabla_H f_t)v, \nabla f_t\right>\\
&+\nabla^2 f_t(\nabla^2 f_t(w),v)-\nabla^2 f_t(V, w)\nabla^2 f_t(V,v)\\
&-\frac{1}{2}\left<\mathcal J\nabla f_t,w\right>\nabla^2 f_t(V,v)+\frac{1}{2}\left<\nabla f_t,V\right>\nabla^2 f_t(\mathcal J(w) ,v)\\
&-\frac{1}{2}\left<\nabla^2 f_t(w),V\right>\left<\mathcal J\nabla f_t,v\right>+\frac{1}{4}\left<\nabla f_t,\mathcal J(w)\right>\left<\mathcal J\nabla f_t,v\right>\\
&-\frac{1}{2}\left<\nabla f_t,V\right>\left<(\nabla_w \mathcal J)\nabla f_t,v\right>-\frac{1}{2}\left<\nabla f_t,V\right>\left<\mathcal J\nabla^2 f_t(w),v\right>=0. 
\end{split}
\]

By setting $v=v_j(t)$ and $w=v_i(t)$, 
\[
\begin{split}
&\nabla^2\dot f_t(v_i(t), v_j(t))+\nabla^3 f_t(\nabla_H f_t,v_i(t),v_j(t))+R_{ij}(t)\\
&+b\left<\Rm(v_i(t),\nabla_H f_t)V, v_j(t)\right>+\sum_{k=0}^{2n}S_{ik}S_{kj}-S_{i0}S_{0j}\\
&-\frac{a}{2}\delta_{2i}S_{0j} +\frac{b}{2}\sum_{k=0}^{2n}J_{ik}S_{kj}-\frac{a}{2}S_{i0}\delta_{2j}-\frac{a^2}{4}\delta_{2i}\delta_{2j}\\
&-\frac{b}{2}\left<(\nabla_{v_i(t)} \mathcal J)\nabla f_t,v_j(t)\right>-\frac{b}{2}\sum_{k=0}^{2n}S_{ik}J_{kj}=0,
\end{split}
\]
where $R_{ij}(t)=\left<\Rm(v_i(t),\nabla_H f_t)\nabla_H f_t,v_j(t)\right>$. 

Next, we apply Proposition \ref{ind} and \ref{Rm} to the above equation. 
\[
\begin{split}
&\nabla^2\dot f_t(v_i(t), v_j(t))+\nabla^3 f_t(\nabla_H f_t,v_i(t),v_j(t))+R_{ij}(t)-\frac{ab}{4}\delta_{0i}\delta_{1j}\\
&+\sum_{k=1}^{2n}S_{ik}S_{kj}-\frac{a}{2}\delta_{2i}S_{0j}+\frac{b}{2}\sum_{k=0}^{2n}J_{ik}S_{kj}-\frac{a}{2}S_{i0}\delta_{2j}\\
&-\frac{a^2}{4}\delta_{2i}\delta_{2j}+\frac{b^2}{4}\left<v_i(t)_{\hor},v_j(t)\right>-\frac{ab}{4}\delta_{i1}\delta_{0j}-\frac{b}{2}\sum_{k=0}^{2n}S_{ik}J_{kj}=0, 
\end{split}
\]
where $v_i(t)_\hor$ denotes the horizontal part of $v_i(t)$. 

In the matrix notations, 
\[
\begin{split}
&\dot S(t)-WS(t)-S(t)W^T+R(t)+C+S(t)DS(t)\\
&-\frac{a}{2}GS(t)+\frac{b}{2}JS(t)-\frac{a}{2}S(t)G^T+\frac{b}{2}S(t)J^T=0, 
\end{split}
\]
where $C=\left(\begin{array}{cc}
C_1 & O\\
O & \frac{b^2}{4}I_{2n-2}
\end{array}\right)$, $G=\left(\begin{array}{cc}
G_1 & O\\
O & O_{2n-2}
\end{array}\right)$, 

$C_1=\left(\begin{array}{ccc}
0 & -\frac{ab}{4} & 0 \\
-\frac{ab}{4} & \frac{b^2}{4} & 0 \\
0 & 0 & -\frac{a^2}{4}+\frac{b^2}{4}
\end{array}\right)$, and $G_1=\left(\begin{array}{cccc}
0 & 0 & 0 \\
0 & 0 & 0 \\
1 & 0 & 0 
\end{array}\right)$. 

On the other hand, it follows from Proposition \ref{Rm} that 
\[
R(t)=\bar R(t)+\left(\begin{array}{cccc}
\frac{1}{4}|\nabla_H f_t|^2_{\varphi_t} & 0 & 0 & 0\\
0 & 0 & 0 & 0\\
0 & 0 & -\frac{3}{4}|\nabla_H f_t|^2_{\varphi_t} & 0\\
0 & 0 & 0 & O
\end{array}\right). 
\]

Therefore, the result follows. 
\end{proof}

\begin{lem}\label{asy}
The matrix $S(t)$ satisfies 
\[
S(1-t)=\frac{1}{t^3}S^{(-3)}+\frac{1}{t^2}S^{(-2)}+\frac{1}{t}S^{(-1)}+o(1), 
\]
as $t\to 0$, where 
$S^{(-3)}=\left(\begin{array}{cccc}
-\frac{12}{a^2} & 0 & 0 & 0\\
0 & 0 & 0 & 0\\
0 & 0 & 0 & 0\\
0& 0 & 0 & O_{2n-2} 
\end{array}\right)$, $S^{(-2)}=\left(\begin{array}{cccc}
0 & 0 & \frac{6}{a} & 0\\
0 & 0 & 0 & 0\\
\frac{6}{a} & 0 & 0 & 0\\
0& 0 & 0 & O_{2n-2}\\
\end{array}\right)$, and $S^{(-1)}=\left(\begin{array}{cccc}
c & -\frac{b}{a} & 0 & 0\\
-\frac{b}{a} & -1 & 0 & 0\\
0 & 0 & -4 & 0\\
0& 0 & 0 & -I_{2n-2}\\
\end{array}\right)$ for some constant $c$ independent of time. 
\end{lem}

\begin{proof}
Let $T(t)=-S(1-t)^{-1}$. Then $T$ satisfies $T(0)=0$ and 
\[
\begin{split}
&\dot T(t)=T(t)U+U^TT(t)+T(t)(\bar R(1-t)+E)T(t)+D.
\end{split}
\]
$\bar R(1-t)=\bar R^{(0)}+t\bar R^{(1)}+t^2\bar R^{(2)}+o(t^3)$. 

A computation shows that $T(t)=tT^{(1)}+t^2T^{(2)}+t^3T^{(3)}+o(t^4)$ as $t\to 0$, where 
$T^{(1)}=D$, 
\[
T^{(2)}=\frac{1}{2}(T^{(1)}U+U^TT^{(1)})=\left(\begin{array}{cccc}
0 & 0 & \frac{a}{2} & 0\\
0 & 0 & 0 & 0\\
\frac{a}{2} & 0 & 0 & 0\\
0 & 0 & 0 & O_{2n-2}
\end{array}\right), 
\]
\[
\begin{split}
&T^{(3)}=\frac{1}{3}(T^{(2)}U+U^TT^{(2)}+T^{(1)}(\bar R(0)+E)T^{(1)})\\
&=\left(\begin{array}{cccc}
\frac{a^2}{3} & -\frac{ab}{12} & 0 & 0\\
-\frac{ab}{12} & \frac{b^2}{12} & 0 & 0\\
0 & 0 & -\frac{a^2}{2}+\frac{b^2}{12}+\frac{\bar R_{22}(1)}{3} & \frac{1}{3}\bar R_{23}(1)\\
0 & 0 & \frac{1}{3}\bar R_{23}(1) & \frac{b^2}{12}I_{2n-2}+\frac{1}{3}\bar R_{33}(1)
\end{array}\right), 
\end{split}
\] 
and 
\[
\bar R(t)=\left(\begin{array}{ccc}
O_{2} & O & O\\
O & \bar R_{22}(t) & \bar R_{23}(t)\\
O & \bar R_{23}(t) & \bar R_{33}(t)\\
\end{array}\right).
\]
Here $\bar R_{33}(t)$ is of size $(2n-2)\times(2n-2)$. 

Since $tT^{(1)}+t^2T^{(2)}+t^3T^{(3)}$ is invertible for $t$ small enough, the result follows. 
\end{proof}

Let $S(t)=\left(\begin{array}{cc}
S_1(t) & S_2(t)\\
S_2(t)^T & S_3(t)
\end{array}\right)$, where $S_1(t)$ is of size $3\times 3$.Then $S_1(t)$ and $S_3(t)$ satisfy
\begin{equation}\label{S1}
\begin{split}
&\dot S_1(t)-U_1S_1(t)-S_1(t)(U_1)^T+\bar R_1(t)\\
&+E_1+S_1(t)D_1 S_1(t)+S_2(t)S_2(t)^T = 0
\end{split}
\end{equation}
and 
\begin{equation}\label{S3}
\begin{split}
&\dot S_3(t)-U_4S_3(t)-S_3(t)U_4^T+\bar R_{33}(t)+E_3+S_2^TD_1S_2+S_3^2=0. 
\end{split}
\end{equation}

Then $s_3(t)=\tr(S_3(t))$ satisfies 
\begin{equation}\label{s3}
\dot s_3(t)+\tr(\bar R_3(t))+\frac{1}{2n-2}s_3(t)^2\leq 0. 
\end{equation}

Let $S_1(t)=\left(\begin{array}{ccc}
S_{11}(t) & S_{12}(t) & S_{13}(t)\\
S_{12}(t) & S_{22}(t) & S_{23}(t)\\
S_{13}(t) & S_{23}(t) & S_{33}(t)\\
\end{array}\right)$. Then $S_{22}(t)$ satisfies 
\[
\dot S_{22}(t)\leq-S_{22}(t)^2-S_{23}(t)^2+bS_{23}(t)-\frac{b^2}{4}\leq -S_{22}(t)^2. 
\]

Since $S_{22}(1-t)=-\frac{1}{t}+o(1)$ as $t\to 0$ and $-\frac{1}{1-t}$ satisfies the equality case of the above inequality. It follows that 
\begin{equation}\label{S22}
S_{22}(t)\geq -\frac{1}{1-t}.
\end{equation} 

By (\ref{S1}), the matrix $\tilde S(t)=\left(\begin{array}{cc}
S_{11}(t)  & S_{13}(t)\\
S_{13}(t) & S_{33}(t)\\
\end{array}\right)$ satisfies 
\[
\dot{\tilde S}(t)\leq -\tilde S(t)C_0\tilde S(t)+C_1\tilde S(t)+\tilde S(t)C_1^T+C_2-\bar R_{22}(t)C_0,
\]
where $C_0=\left(\begin{array}{cc}
0  & 0\\
0 & 1
\end{array}\right)$, 
$C_1=\left(\begin{array}{cc}
0 & -a/2\\
a & 0
\end{array}\right)$, and $C_2=\left(\begin{array}{cc}
-a^2/4 & 0\\
0 & a^2
\end{array}\right)$. 

Let $C_3=\left(\begin{array}{cc}
0 & a/2\\
a/2 & 0
\end{array}\right)$ and $C_4=\left(\begin{array}{cc}
0 & 0\\
a & 0
\end{array}\right)$. Then 
\[
\begin{split}
&\frac{d}{dt}\left(\tilde S(t) +C_3\right)\leq -(\tilde S(t) +C_3)C_0(\tilde S(t) +C_3)\\
&+C_4(\tilde S(t) +C_3)+(\tilde S(t) +C_3)C_4^T-\bar R_{22}(t)C_0. 
\end{split}
\]

By combining this with (\ref{s3}) and using the assumption that the Tanaka-Webster Ricci curvature is non-negative, we obtain 
\[
\begin{split}
&\frac{d}{dt}\left(\tilde S(t) +C_3+s_3C_0\right)\\
&\leq -(\tilde S(t)+C_3)C_0(\tilde S(t)+C_3)-\frac{1}{2n-2}s_3^2C_0\\
&+C_4(\tilde S(t)+C_3+s_3C_0)+(\tilde S(t)+C_3+s_3C_0)C_4^T. 
\end{split}
\]

It follows that 
\[
\begin{split}
&\frac{d}{dt}\left(\tilde S(t)+C_3+s_3C_0\right)\leq (\tilde S(t)+C_3)C_0K +KC_0(\tilde S(t) +C_3)+KC_0K\\
&+C_4(\tilde S(t) +C_3+s_3C_0)+(\tilde S(t) +C_3+s_3C_0)C_4^T-\frac{1}{2n-2}s_3^2C_0
\end{split}
\]
where $K$ is a matrix. 

Let $\bar S=\tilde S(t)+C_3+s_3C_0$. Then the above equation becomes
\begin{equation}\label{SCC}
\begin{split}
&\frac{d}{dt}\bar S(t)\leq \bar S(t)(C_0K+C_4^T) +(KC_0+C_4)\bar S(t)\\
&+KC_0K-s_3C_0K-s_3KC_0-\frac{1}{2n-2}s_3^2C_0. 
\end{split}
\end{equation}

Let $K=\left(\begin{array}{cc}
\frac{12}{a^2(1-t)^3} & -\frac{6}{a(1-t)^2}\\
 -\frac{6}{a(1-t)^2}  & \frac{4}{1-t}
\end{array}\right)$. Since $\left(\begin{array}{cc}
\frac{72(n-1)}{ca^2(1-t)^4} & \frac{6s_3}{a(1-t)^2}\\
\frac{6s_3}{a(1-t)^2} & \frac{cs_3^2}{2n-2}
\end{array}\right)$ with $c>0$ is non-negative definite, it follows that 
\[
\begin{split}
&\frac{d}{dt}\bar S(t)\leq \bar S(t)(C_0K+C_4^T) +(KC_0+C_4)\bar S(t)\\
&+KC_0K-\frac{s_38}{1-t}C_0-\frac{1-c}{2n-2}s_3^2C_0+\frac{72(n-1)}{ca^2(1-t)^4}C_5\\
&\leq \bar S(t)(C_0K+C_4^T) +(KC_0+C_4)\bar S(t)\\
&+KC_0K+\frac{32(n-1)}{(1-c)(1-t)^2}C_0+\frac{72(n-1)}{ca^2(1-t)^4}C_5, 
\end{split}
\]
where $C_5=\left(\begin{array}{cc}
1 & 0\\
0 & 0
\end{array}\right)$ and $1>c>0$. 

On the other hand, 
\[
S_0(t)=\left(\begin{array}{cc}
\frac{12(26-26n+5c-6nc+c^2)}{(1-c)ca^2(1-t)^3} & -\frac{6(14-14n+c-2nc+c^2)}{(1-c)ca(1-t)^2}\\
-\frac{6(14-14n+c-2nc+c^2)}{(1-c)ca(1-t)^2} & -\frac{4[(c+6)(n-1)+c(n-c)]}{(1-c)c(1-t)}
\end{array}\right)
\]
satisfies 
\[
\begin{split}
&\frac{d}{dt}S_0(t) = S_0(t)(C_0K+C_4^T) +(KC_0+C_4)S_0(t)\\
&+KC_0K+\frac{32(n-1)}{(1-c)(1-t)^2}C_0+\frac{72(n-1)}{ca^2(1-t)^4}C_5. 
\end{split}
\]

Therefore, $F(t)=\bar S(1-t)-S_0(1-t)$ satisfies 
\[
\dot F\geq -F(C_0K-C_4^T) -(KC_0-C_4)F. 
\]

A computation using Lemma \ref{asy} and definition of $S_0(t)$ shows that $F\geq 0$ for $t$ small enough. It follows that $F\geq 0$ on $[0,1]$ (see \cite[Proposition 1]{Ro}). Therefore, $\bar S\geq S_0$. Hence, (\ref{Lap}) follows from this and (\ref{S22}). 
\end{proof}

We finish this section by recording the following Laplacian comparison theorem. It is an immediate consequence of (\ref{Lap}). 

\begin{cor}\label{LapCom}
Under the assumptions of Theorem \ref{main}, the following holds on the set where $f$ is twice differentiable:
\[
\Delta_H f\leq 2N,
\]
where $f:M\to\Real$ is defined by $f(x)=d^2_S(x_0,x)$. 
\end{cor}

\smallskip

\section{The Riemannian case}\label{Riem}

In this section, we consider the case when the distance function $d_\e$ of the metric measure space is given by the Riemannian metric $\left<\cdot,\cdot\right>^\e$ extending the contact sub-Riemannian one and satisfying the conditions 
\[
|V|^\e=\e\quad \text{and}\quad \left<V,X\right>^\e=0
\]
for all horizontal vector fields $X$. 

The gradient $\nabla^\e f$ of $f$ with respect to the metric $\left<\cdot,\cdot\right>^\e$ is given by 
\[
\nabla^\e f=\nabla_H f+\frac{1}{\e^2}\left<\nabla f,V\right>V=\nabla f+\frac{1-\e^2}{\e^2}\left<\nabla f,V\right>V. 
\]
Let $f_t^\e(x)=-\frac{1}{2(1-t)}d^2_\e(x_0,x)$. It satisfies the following Hamilton-Jacobi equation
\begin{equation}\label{HJe}
\begin{split}
&\dot f_t^\e+\frac{1}{2}\left(|\nabla f^\e|^2+\frac{1-\e^2}{\e^2}\left<\nabla f^\e,V\right>^2\right)\\
&=\dot f_t^\e+\frac{1}{2}\left(|\nabla_H f^\e|^2+\frac{1}{\e^2}\left<\nabla f^\e,V\right>^2\right)=0.
\end{split}
\end{equation}

For each point $x$ in the set where $d^2_\e$ is twice differentiable, let $t\mapsto \varphi_t^\e(x)$ be the unique geodesic starting from $x$ and ending at $x_0$. By the method of characteristics, 
\begin{equation}\label{flowe}
\dot\varphi_t^\e(x)=\left(\nabla f_t^\e+\frac{1-\e^2}{\e^2}\left<\nabla f_t^\e,V\right>V\right)(\varphi^\e_t(x)). 
\end{equation}

By the same arguments as in Section \ref{Sub}, in order to show the measure contraction property $MCP(0,N)$, it is enough to prove the following condition holds $\vol$ a.e.
\begin{equation}\label{vole}
\det(d\varphi_t^\e)\geq (1-t)^N. 
\end{equation}

We first show the following. 

\begin{lem}\label{masse}
Let $x$ be a point where $\varphi_t$ is differentiable. Then 
\[
\det(\varphi_t^\e)=\exp\left(\int_0^t\Delta f^\e_s(\varphi_s)+\frac{1-\e^2}{\e^2}\left<\nabla^2 f_s^\e(V),V\right>_{\varphi_s}ds\right)
\]
holds at $x$. 
\end{lem}

\begin{proof}
Let $v_0^\e(t),...,v_{2n}^\e(t)$ be the family of orthogonal frames such that $v_0^\e(t)=V(\varphi_t^\e(x))$ and $\left<\dot v_i^\e(t), v_j^\e(t)\right>=0$ for any $i, j\neq 0$. Let $A^\e(t)$ be the matrix defined by 
\[
d\varphi_t^\e(v^\e(0))=A^\e(t)v^\e(t), 
\]
where $v^\e(t)=\left(v_1^\e(t),...,v_{2n}^\e(t)\right)^T$. 

It follows from (\ref{flowe}) that 
\[
\begin{split}
&A^\e(t)\Big(\nabla^2 f_t^\e(v(t))+\frac{1-\e^2}{\e^2}\left<\nabla^2 f_t^\e(v(t)),V\right>V(\varphi_t)\\
&+\frac{1-\e^2}{\e^2}\left<\nabla f_t^\e, \nabla V(v(t))\right>V(\varphi_t)+\frac{1-\e^2}{\e^2}\left<\nabla f_t^\e,V\right>_{\varphi_t}\nabla V(v(t))\Big)\\
&=\frac{D}{dt}d\varphi_t^\e(v(0))=\dot A^\e(t)v(t)+A^\e(t)Wv(t).  
\end{split}
\]
Therefore, 
\[
\Delta f^\e_t(\varphi_t)+\frac{1-\e^2}{\e^2}\left<\nabla^2 f_t^\e(V),V\right>_{\varphi_t}=\tr(A^\e(t)^{-1}\dot A^\e(t)). 
\]
Hence, 
\[
\det(A^\e(t))=\exp\left(\int_0^t\Delta f^\e_s(\varphi_s)+\frac{1-\e^2}{\e^2}\left<\nabla^2 f_s^\e(V),V\right>_{\varphi_s}ds\right)
\]
as claimed. 
\end{proof}

Next, we assume that the contact manifold is Sasakian. By differentiating (\ref{HJe}) and applying Proposition \ref{ind}, 
\begin{equation}\label{dHJe}
\begin{split}
&\nabla\dot f_t^\e+\nabla^2 f^\e_t(\nabla^\e f_t^\e)+\frac{1-\e^2}{2\e^2}\left<\nabla f^\e_t,V\right>\cJ\nabla f_t^\e=0. 
\end{split}
\end{equation}
It follows that $\frac{d}{dt}\left<\nabla f_t^\e,V\right>_{\varphi_t}=0$ and $\frac{d}{dt}|\nabla f_t^\e|^2_{\varphi_t}=0$. Let $a:=|\nabla_H f_0^\e|_{x}$ and let $b=\left<\nabla f_0^\e,V\right>_x$. 

\begin{lem}\label{framee}
There is a family of orthonormal frame $\{v_0^\e(t), ..., v_{2n}^\e(t)\}$ defined along the path $t\mapsto\varphi_t^\e(x)$ such that $v_0(t)=V(\varphi_t)$, $v_{2i}(t)=\cJ v_{2i-1}(t)$, and $\dot v(t)=W^\e v(t)$, where  $i=1,...,n$, 
\[
W^\e=\left(
       \begin{array}{cc}
         W_1^\e & O \\
         O  & O_{2n-2} \\
       \end{array}
     \right)\text{ and } 
W^\e_1=\left(\begin{array}{ccc}
0 & 0 & -\frac{a}{2}\\
0 & 0 & -\left(\frac{1}{2\e^2}-1\right)b\\
\frac{a}{2} & \left(\frac{1}{2\e^2}-1\right)b & 0
\end{array}\right). 
\]
\end{lem}

\begin{proof}
Let $v_0^\e(t)=V(\varphi_t^\e)$, $v_1^\e(t)=\frac{1}{a}\nabla_H f_t^\e(\varphi_t^\e)$, and $v_2^\e(t)=\frac{1}{a}\cJ \nabla f_t^\e(\varphi_t^\e)$. Then, by Proposition \ref{ind} and (\ref{dHJe}), 
\[
\begin{split}
\frac{D}{dt} v_0^\e(t)&=\nabla V(\dot\varphi_t^\e)=-\frac{1}{2}\cJ \nabla_H f_t^\e(\varphi_t^\e)=-\frac{a}{2}v_2^\e(t), \\
\frac{D}{dt} v_1^\e(t)&=\frac{1}{a}\nabla \dot f_t^\e(\varphi_t^\e)+\frac{1}{a}\nabla^2 f_t^\e(\dot\varphi_t^\e)-\frac{b}{a}\nabla V(\dot \varphi_t^\e)\\
&=-\frac{(1-\e^2)b}{2a\e^2}\cJ\nabla f_t^\e(\varphi_t^\e)+\frac{b}{2a}\cJ\nabla f_t(\varphi_t^\e)\\
&=\left(1-\frac{1}{2\e^2}\right)b v_2^\e(t),\\
\frac{D}{dt} v_2^\e(t)&=\frac{1}{a}\cJ\nabla \dot f^\e_t(\varphi_t^\e)+\frac{1}{a}(\nabla_{\nabla^\e f^\e_t}\cJ) \nabla f_t^\e(\varphi_t^\e)+\frac{1}{a}\cJ \nabla^2 f_t^\e(\dot \varphi_t^\e)\\
&=\frac{a}{2}v_0^\e(t)+\left(\frac{1}{2\e^2}-1\right)bv_1^\e(t). 
\end{split}
\]

Finally, we extend $v_0^\e(0), v_1^\e(0), v_2^\e(0)$ to an orthonormal frame 
\[
v_0^\e(0), ..., v_{2n}^\e(0)
\] 
such that $v_{2i}(0)=\cJ v_{2i-1}(0)$, where $i=1, ..., n$. Let $v_j(t)$ be the parallel transport of $v_j^\e(0)$ along the path $t\mapsto\varphi_t^\e(x)$, where $j=3,...,2n$. For each such fixed $j$, it follows that 
\[
\frac{d}{dt}\left<v_j^\e(t), v_k^\e(t)\right>=\sum_{m=0}^2(W^\e_1)_{km}\left<v_j^\e(t), v_m^\e(t)\right>
\]
for $k=0,1,2$. 
Therefore, $\left<v_j^\e(t), v_k^\e(t)\right>=0$ and the result follows. 
\end{proof}

\begin{lem}\label{Ricce}
Let $S^\e(t)$ be the matrix defined by 
\[
S_{ij}^\e(t)=\left<\nabla^2 f_t^\e(v_i^\e(t)),v_j^\e(t)\right>. 
\]
Then $S^\e(t)$ satisfies the following matrix Riccati equation 
\[
\begin{split}
&\dot S^\e(t)-U^\e S^\e(t)-S^\e(t)(U^\e)^T+\bar R^\e(t)+E^\e+S^\e(t)D^\e S^\e(t)=0,
\end{split}
\]
with the condition $\lim_{t\to 1}S^\e(t)^{-1}=0$, where 
\[
\begin{split}
&\bar R_{ij}^\e(t)=\left<\overline{Rm}(v_i^\e(t)),\nabla_H f^\e_t)\nabla_H f^\e_t,v_j^\e(t)\right>, 
U^\e=\left(\begin{array}{cc}
U_1^\e & O\\
O & \frac{(1-\e^2)b}{2\e^2}J_{2n-2}
\end{array}\right), \\
&E^\e=\left(\begin{array}{cc}
E_1^\e & O\\
O & \frac{b^2}{4}I_{2n-2}
\end{array}\right), D=\left(\begin{array}{cc}
D_1^\e & O\\
O & I_{2n-2}
\end{array}\right),\\
&E_1^\e=\left(\begin{array}{ccc}
\frac{a^2}{4} & -\frac{ab}{4} & 0\\
-\frac{ab}{4} & \frac{b^2}{4} & 0\\
0 & 0 & \frac{b^2}{4}-a^2+\frac{a^2}{4\e^2}
\end{array}\right), \\
&U_1^\e=\left(\begin{array}{ccc}
0 & 0 & -\frac{a}{2}\\
0 & 0 & \frac{b}{2}\\
-\frac{(1-2\e^2)a}{2\e^2} & -\frac{b}{2} & 0
\end{array}\right), D_1^\e=\left(\begin{array}{ccc}
\frac{1}{\e^2} & 0 & 0\\
0 & 1 & 0\\
0 & 0 & 1
\end{array}\right), 
\end{split}
\]
and $J_{2n-2}$ is the $(2n-2)\times (2n-2)$ matrix such that the $ij$-th entry is given by $\left<\cJ(v_{i+2}(t),v_{j+2}(t)\right>$. 
\end{lem}

\begin{proof}
By differentiating (\ref{dHJe}), we obtain 
\[
\begin{split}
&\nabla^2\dot f_t^\e(v, w)+\nabla^3 f^\e_t(w, \nabla^\e f_t^\e,v)+\nabla^2 f^\e_t(\nabla^2 f_t^\e(w),v)\\
&+\frac{1-\e^2}{\e^2}\left(\nabla^2 f^\e_t(V,w)+\frac{1}{2}\left<\cJ \nabla f_t^\e,w\right>\right)\left(\nabla^2 f^\e_t(V,v)+\frac{1}{2}\left<\cJ \nabla f_t^\e,v\right>\right)\\
&-\frac{1-\e^2}{2\e^2}\left<\nabla f^\e_t,V\right>\Big(\nabla^2 f^\e_t(\cJ w,v)-\left<(\nabla_w \cJ) \nabla f_t^\e,v\right>-\left<\cJ \nabla^2 f_t^\e(w),v\right>\Big)=0. 
\end{split}
\]

By the Ricci identity and  Proposition \ref{Rm}, 
\[
\begin{split}
&\nabla^2\dot f_t^\e(v, w)+\nabla^3 f^\e_t(\nabla^\e f_t^\e, w,v)+\left<\Rm(w,\nabla^\e f_t^\e)\nabla^\e f_t^\e, v\right>+\nabla^2 f^\e_t(\nabla^2 f_t^\e(w),v)\\
&-\frac{(1-\e^2)\left<\nabla f_t^\e,V\right>}{4\e^2}\left(\frac{1}{\e^2}\left<\nabla f_t^\e,V\right>\left<w_\hor,v\right>-\left<w,V\right>\left<\nabla_H f_t^\e,v\right>\right)\\
&+\frac{1-\e^2}{\e^2}\left(\nabla^2 f^\e_t(V,w)+\frac{1}{2}\left<\cJ \nabla f_t^\e,w\right>\right)\left(\nabla^2 f^\e_t(V,v)+\frac{1}{2}\left<\cJ \nabla f_t^\e,v\right>\right)\\
&-\frac{1-\e^2}{2\e^2}\left<\nabla f^\e_t,V\right>\Big(\nabla^2 f^\e_t(\cJ w,v)-\left<(\nabla_w \cJ) \nabla f_t^\e,v\right>-\left<\cJ \nabla^2 f_t^\e(w),v\right>\Big)=0. 
\end{split}
\]

After rewriting this in matrix notation, we obtain 
\begin{equation}\label{Riccati1}
\begin{split}
&\dot S^\e(t)-W^\e S^\e(t)-S^\e(t)(W^\e)^T+R^\e(t)+\frac{(1-\e^2)a b }{4\e^2}F\\
&+\frac{1-\e^2}{\e^2}\left(S^\e GS^\e+\frac{a}{2}PS^\e+\frac{a}{2}S^\e P^T+\frac{a^2}{4}H\right)\\
&-\frac{(1-\e^4)b^2}{4\e^4}D+S^\e(t)^2 -\frac{(1-\e^2)b}{2\e^2}\Big(JS^\e (t)+S^\e(t)J^T\Big)=0, 
\end{split}
\end{equation}
where $F=\left(\begin{array}{cc}
F_1 & O\\
O & O_{2n-2}
\end{array}\right)$, $G=\left(\begin{array}{cc}
G_1 & O\\
O & O_{2n-2}
\end{array}\right)$, $H=\left(\begin{array}{cc}
H_1 & O\\
O & O_{2n-2}
\end{array}\right)$, $P=\left(\begin{array}{cc}
P_1 & O\\
O & O_{2n-2}
\end{array}\right)$, $F_1=\left(\begin{array}{ccc}
0 & 1 & 0\\
1 & 0 & 0\\
0 & 0 & 0
\end{array}\right)$, $G_1=\left(\begin{array}{ccc}
1 & 0 & 0\\
0 & 0 & 0\\
0 & 0 & 0
\end{array}\right)$, \\$H_1=\left(\begin{array}{ccc}
0 & 0 & 0\\
0 & 0 & 0\\
0 & 0 & 1
\end{array}\right)$, $P_1=\left(\begin{array}{ccc}
0 & 0 & 0\\
0 & 0 & 0\\
1 & 0 & 0
\end{array}\right)$. $J$ and $R^\e(t)$ are the matrices with $ij$-th entry equal to $\left<\cJ(v_i^\e(t)),v_j^\e(t)\right>$ and $\left<\Rm(v_i(t),\nabla f^\e_t)\nabla f^\e_t,v_j(t)\right>$, respectively. 

By Proposition \ref{Rm}, we have, for all $i, j\neq 0$, 
\[
\begin{split}
&\left<\Rm(v_i(t),\nabla^\e f_t^\e)\nabla^\e f_t^\e,v_j(t)\right>\\
&=\left<\Rm(v_i(t),\nabla_H f_t^\e)\nabla_H f_t^\e,v_j(t)\right>+\frac{b^2}{\e^4}\left<\Rm(v_i(t),V)V,v_j(t)\right>\\
&=\left<\Rm(v_i(t),\nabla_H f_t^\e)\nabla_H f_t^\e,v_j(t)\right>+\frac{b^2}{4\e^4}\delta_{ij}\\
&=\left<\overline{\Rm}(v_i(t),\nabla_H f_t^\e)\nabla_H f_t^\e,v_j(t)\right>-\frac{3a^2\delta_{2i}\delta_{2j}}{4}+\frac{b^2}{4\e^4}\delta_{ij}\\
&=\left<\overline{\Rm}(v_i(t),\nabla^\e f_t^\e)\nabla^\e f_t^\e,v_j(t)\right>-\frac{3a^2\delta_{2i}\delta_{2j}}{4}+\frac{b^2}{4\e^4}\delta_{ij}, \\
&\left<\Rm(v_i(t),\nabla^\e f_t^\e)\nabla^\e f_t^\e,v_0(t)\right>=\left<\Rm(v_i(t),\nabla_H f_t^\e)\nabla_H f_t^\e,v_0(t)\right>\\
&+\frac{b}{\e^2}\left<\Rm(v_i(t),V)\nabla_H f_t^\e,v_0(t)\right>=-\frac{ab}{4\e^2}\delta_{i1}, \\
&\left<\Rm(v_0(t),\nabla^\e f_t^\e)\nabla^\e f_t^\e,v_0(t)\right>=\left<\Rm(v_0(t),\nabla_H f_t^\e)\nabla_H f_t^\e,v_0(t)\right>=\frac{1}{4}a^2. 
\end{split}
\]

Therefore, $R^\e(t)=\bar R^\e(t)+K$, where $K^\e=\left(\begin{array}{cc}
K^\e_1 & O\\
O & \frac{b^2}{4\e^4}I_{2n-2}
\end{array}\right)$ and $K_1^\e=\left(\begin{array}{ccc}
\frac{a^2}{4} & -\frac{ab}{4\e^2} & 0\\
-\frac{ab}{4\e^2} & \frac{b^2}{4\e^4} & 0\\
0 & 0 & \frac{b^2}{4\e^4}-\frac{3a^2}{4}
\end{array}\right)$. 

The result follows from this and (\ref{Riccati1}). 
\end{proof}

\begin{lem}\label{asye}
The matrix $S^\e(t)$ satisfies 
\[
S^\e(1-t)=\frac{1}{t}S^{\e,(-1)}+S^{\e,(0)}+tS^{\e,(1)}+o(t)
\] 
as $t\to 0$, where 
\[
S^{\e,(-1)}=-\left(\begin{array}{cccc}
\e^2 & 0 & 0 & 0\\
0 & 1 & 0 & 0\\
0 & 0 & 1& 0\\
0 & 0 & 0 & 1
\end{array}\right), S^{\e,(0)}=-\left(\begin{array}{cccc}
0 & 0 & \frac{(1-\e^2)a}{2} & 0\\
0 & 0 & 0 & 0\\
\frac{(1-\e^2)a}{2} & 0 & 0 & 0\\
0 & 0 & 0 & 0
\end{array}\right), 
\]
\[
S^{\e,(1)}=\left(\begin{array}{cccc}
\frac{a^2\e^4}{12} & -\frac{\e^2ab}{12} & 0 & 0\\
-\frac{\e^2ab}{12} & \frac{b^2}{12} & 0 & 0\\
0 & 0 & \frac{b^2}{12}-\frac{a^2\e^2}{4}+\frac{\bar R_{22}(1)}{3} & \frac{\bar R_{23}(1)}{3}\\
0 & 0 & \frac{\bar R_{23}(1)}{3} & \frac{b^2}{12}+\frac{\bar R_{33}(1)}{3}
\end{array}\right), 
\] 
$\bar R^\e(t)=\left(\begin{array}{ccc}
0 & 0 & 0\\
0 & \bar R^\e_{22}(t) & \bar R^\e_{23}(t)\\
0 & \bar R^\e_{23}(t)^T & \bar R^\e_{33}(t)
\end{array}\right)$, $\bar R^\e_{22}(t)$ is a $1\times 1$ block, and $\bar R^\e_{33}(t)$ is a $(2n-1)\times (2n-1)$ block. 
\end{lem}

\begin{proof}
Let $T^\e(t)=-S^\e(1-t)^{-1}$. It satisfies the following equation 
\[
\begin{split}
&\dot T^\e(t)=-S^\e(1-t)^{-1}\dot S^\e(1-t)S^\e(1-t)^{-1}\\
&=-S^\e(1-t)^{-1}U^\e-(U^\e)^TS^\e(1-t)^{-1}\\
&+S^\e(1-t)^{-1}(\bar R^\e(1-t)+E^\e)S(1-t)^{-1}+D^\e \\
&=T^\e(t) U+U^TT^\e(t)+T^\e(t)(\bar R^\e(1-t)+E^\e)T^\e(t)+D^\e
\end{split}
\]
It follows that $T^\e(t)=tT^{\e,(1)}+t^2T^{\e,(2)}+t^3T^{\e,(3)}+o(t^4)$ at $t\to 0$, where $T^{\e,(1)}=D^\e$, $T^{\e,(2)}=\frac{1}{2}(T^{\e,(1)}U^\e+(U^\e)^TT^{\e,(1)})=\left(\begin{array}{cccc}
0 & 0 & -\frac{(1-\e^2)a}{2\e^2} & 0\\
0 & 0 & 0 & 0\\
-\frac{(1-\e^2)a}{2\e^2} & 0 & 0 & 0\\
0 & 0 & 0 & O
\end{array}\right)$, \\
\[
\begin{split}
&T^{\e,(3)}=\frac{1}{3}(T^{\e,(2)}U^\e+(U^\e)^TT^{\e,(2)}+T^{\e,(1)}(\bar R^\e(1)+E^\e)T^{\e,(1)})\\
&=\frac{1}{3}\left(\begin{array}{cccc}
\frac{a^2(3-6\e^2+4\e^4)}{4\e^4} & -\frac{ab}{4} & 0 & 0\\
-\frac{ab}{4} & \frac{b^2}{4} & 0 & 0\\
0 & 0 & \frac{3a^2-6a^2\e^2+b^2\e^2}{4\e^2}+\bar R^\e_{22}(1) & \bar R^\e_{23}(1)\\
0 & 0 & \bar R^\e_{23}(1) & \frac{b^2}{4}I_{2n-2}+\bar R^\e_{33}(1)
\end{array}\right). 
\end{split}
\] 
The result follows from this. 
\end{proof}

Finally, we give the proof of Theorem \ref{main2}. 

\begin{proof}[Proof of Theorem \ref{main2}]
Let $S^\e(t)=\left(\begin{array}{cc}
S_1^\e(t) & S_2^\e(t)\\
S_2^\e(t)^T & S_3^\e(t)
\end{array}\right)$, where $S_1^\e(t)$ is a $3\times 3$ block. It follows from Lemma \ref{Ricce} that 
\begin{equation}\label{S1e}
\dot S^\e_1(t)-U_1^\e S_1^\e(t)-S_1^\e(t)(U_1^\e)^T+\bar R^\e_1(t)+E_1^\e+S_1^\e(t)D_1^\e S_1^\e(t)\leq 0
\end{equation}
and 
\begin{equation}\label{S3e}
\begin{split}
&\dot S^\e_3(t)+\bar R^\e_3(t)+\frac{b^2}{4}I+S^\e_3(t)^2\\
& -\frac{(1-\e^2)b}{2\e^2}\Big(J_{2n-2}S_3(t)+S_3(t)J_{2n-2}^T\Big)\leq 0. 
\end{split}
\end{equation}

Let $s_3^\e$ be the trace of $S_3^\e$. It follows that 
\begin{equation}\label{s3e}
\begin{split}
&\dot s^\e_3(t)+\tr \bar R^\e_3(t)+\frac{1}{2n-2}s^\e_3(t)^2 \leq 0. 
\end{split}
\end{equation}

By (\ref{S1e}), $S^\e_1(t)=\left(\begin{array}{ccc}
S_{11}^\e(t) & S_{12}^\e(t) & S_{13}^\e(t)\\
S_{12}^\e(t) & S_{22}^\e(t) & S_{23}^\e(t)\\
S_{13}^\e(t) & S_{23}^\e(t) & S_{33}^\e(t)\end{array}\right)$ gives 
\[
\dot S_{22}^\e(t)= -\left(S_{23}^\e(t)-\frac{b}{2}\right)^2-\frac{S_{12}^\e(t)^2}{\e^2}-S_{22}^\e(t)^2\leq -S_{22}^\e(t)^2. 
\]
It follows from this and Lemma \ref{asye} that 
\begin{equation}\label{S22e}
S_{22}^\e(1-t)\geq -\frac{1}{t}. 
\end{equation}

By (\ref{S1e}), $\tilde S^\e(t)=\left(\begin{array}{cc}
S_{11}^\e(t) & S_{13}^\e(t)\\
S_{31}^\e(t) & S_{33}^\e(t)
\end{array}\right)$ satisfies 
\[
\begin{split}
&\frac{d}{dt}\left({\tilde S}^\e(t)+C_3\right)+(\tilde S^\e(t)+C_3)\left(C_0+\frac{1}{\e^2}C_5\right)(\tilde S^\e(t)+C_3)\\
&-C_4(\tilde S^\e(t)+C_3)-(\tilde S^\e(t)+C_3)C_4^T+\bar R_{22}^\e(t)C_0\leq 0, 
\end{split}
\]
where $C_5=\left(\begin{array}{cc}
1 & 0\\
0 & 0\end{array}\right)$. 

By combining this with (\ref{s3e}) and using the assumption that the Tanaka-Webster Ricci curvature is non-negative, we see that $\bar S=\tilde S+C_3+s_3C_0$ satisfies 
\[
\begin{split}
0&\geq \frac{d}{dt}\bar S^\e(t)+(\tilde S^\e(t)+C_3)\left(C_0+\frac{1}{\e^2}C_5\right)(\tilde S^\e(t)+C_3)\\
&+\frac{1}{2n-2}(s_3^\e(t))^2C_0-C_4\bar S^\e(t) - \bar S^\e(t) C_4^T\\
&\geq \frac{d}{dt}\bar S^\e(t)-K^\e(t)\left(C_0+\frac{1}{\e^2}C_5\right)(\tilde S^\e(t)+C_3)\\
& -(\tilde S^\e(t)+C_3)\left(C_0+\frac{1}{\e^2}C_5\right)K^\e(t)-K^\e(t)\left(C_0+\frac{1}{\e^2}C_5\right)K^\e(t)\\
&+\frac{1}{2n-2}(s_3^\e(t))^2C_0-C_4\bar S^\e(t) - \bar S^\e(t) C_4^T\\
\end{split}
\]
\[
\begin{split}
&=\frac{d}{dt}\bar S^\e(t)-\left(K^\e(t)C_0+\frac{1}{\e^2}K^\e(t)C_5+C_4\right)\bar S^\e(t)\\
& -\bar S^\e(t)\left(C_0K^\e(t)+\frac{1}{\e^2}C_5K^\e(t)+C_4^T\right)+\frac{1}{2n-2}(s_3^\e(t))^2C_0 \\
&+s_3^\e(t)K^\e(t)C_0 +s_3^\e(t)C_0K^\e(t)-K^\e(t)\left(C_0+\frac{1}{\e^2}C_5\right)K^\e(t),  
\end{split}
\]
where 
\[
\begin{split}
K^\e(t)&=\frac{1}{12+(1-t)^2a^2\e^2}\left(\begin{array}{cc}
\frac{12\e^2}{1-t} & -6a\e^2\\
 -6a\e^2 & \frac{4(t^2a^2\e^2+3)}{1-t}
\end{array}\right)\\
&=\left(\begin{array}{cc}
K_{11}(t) & K_{12}(t)\\
K_{12}(t) & K_{22}(t)
\end{array}\right).
\end{split}
\] 

Next, we use the term $s_3^\e(t)^2$ and estimate 
\[
\begin{split}
&\frac{d}{dt}\bar S^\e(t)-\left(K^\e(t)C_0+\frac{1}{\e^2}K^\e(t)C_5+C_4\right)\bar S^\e(t)\\
& -\bar S^\e(t)\left(C_0K^\e(t)+\frac{1}{\e^2}C_5K^\e(t)+C_4^T\right)-K^\e(t)\left(C_0+\frac{1}{\e^2}C_5\right)K^\e(t)\\
&-\frac{(2n-2)K_{22}^\e(t)(t)^2}{1-c}C_0-\frac{(2n-2)K_{12}^\e(t)(t)^2}{c}C_5 \leq 0, 
\end{split}
\]
where $c$ is a constant in the open interval $(0, 1)$. 

Finally, if we let $\hat S=\left(C_0+\frac{1}{\e} C_5\right)\bar S\left(C_0+\frac{1}{\e} C_5\right)$ and\\ $\hat K=\left(C_0+\frac{1}{\e} C_5\right) K\left(C_0+\frac{1}{\e} C_5\right)$, then 
\begin{equation}\label{Shat}
\begin{split}
&\frac{d}{dt}\hat S-\left(\hat K+\e C_4\right)\hat S -\hat S\left(\hat K+\e C_4\right)^T\\
&-\hat K^2-\frac{(2n-2)\hat K_{22}(t)^2}{1-c}C_0-\frac{(2n-2)\hat K_{12}(t)^2}{c}C_5 \leq  0. 
\end{split}
\end{equation}

Let $\hat S_0$ be a solution of the following equation 
\begin{equation}\label{Shat0}
\begin{split}
&\frac{d}{dt}\hat S_0-\left(\hat K+\e C_4\right)\hat S_0 -\hat S_0\left(\hat K+\e C_4\right)^T\\
&-\hat K^2-\frac{(2n-2)\hat K_{22}(t)^2}{1-c}C_0-\frac{(2n-2)\hat K_{12}(t)^2}{c}C_5=0. 
\end{split}
\end{equation}
Then the matrix $Sr(t)=\hat S_0(1-t)$ satisfies the following equation 
\[
\begin{split}
&\frac{d^3}{dt^3}Sr_{11}+\frac{18(t^2a^2\e^2+4)}{t(12+t^2a^2\e^2)}\frac{d^2}{dt^2}Sr_{11}\\
&+\frac{18(5t^4ae^4+48t^2a^2\e^2+48)}{t^2(12+t^2a^2\e^2)^2}\frac{d}{dt}Sr_{11}\\
&+\frac{24a^2\e^2(5t^2a^2\e^2+24)}{t(12+t^2a^2\e^2)^2}Sr_{11}+\frac{72a^2\e^2}{t^2(12+t^2a^2\e^2)^2}\\
&+\frac{144a^2\e^2(n-1)(13t^4a^4\e^4+3t^4a^4\e^4c+144)}{t^2c(1-c)(12+t^2a^2\e^2)^4}\\
&+\frac{144a^2\e^2(n-1)(132t^2a^2\e^2-36t^2a^2\e^2c)}{t^2c(1-c)(12+t^2a^2\e^2)^4}=0. 
\end{split}
\]

It turns out that the above equation has closed form solutions which is first found by using a computer. Here we describe how the solution can be obtained. Let us first look at the corresponding homogeneous equation. 
\begin{equation}\label{3rd}
\begin{split}
&\frac{d^3}{dt^3}g(t)+\frac{18(t^2a^2\e^2+4)}{t(12+t^2a^2\e^2)}\frac{d^2}{dt^2}g(t)\\
&+\frac{18(5t^4ae^4+48t^2a^2\e^2+48)}{t^2(12+t^2a^2\e^2)^2}\frac{d}{dt}g(t)\\
&+\frac{24a^2\e^2(5t^2a^2\e^2+24)}{t(12+t^2a^2\e^2)^2}g(t)=0. 
\end{split}
\end{equation}

The above equation is the, so called, second symmetric power, of the following second order equation (see \cite{Si} for the precise definition and a detail discussion of the symmetric power). 
\begin{equation}\label{2nd}
\frac{d^2}{dt^2} f(t)+\frac{6(t^2a^2\e^2+4)}{t(12+t^2a^2\e^2)}\frac{d}{dt}f(t)+\frac{6a^2\e^2}{12+t^2a^2\e^2}f(t)=0. 
\end{equation} 
Solutions of the equation (\ref{2nd}) are given by $f(t)=\frac{c_1}{t(12+t^2a^2\e^2)}+\frac{c_2}{12+t^2a^2\e^2}$, where $c_1$ and $c_2$ are constants. We remark here that closed form solutions of second order equations like (\ref{2nd}) can be found, if it exists, using the Kovacic algorithm \cite{Ko}. However, the solutions of (\ref{2nd}) can be found by a simple calculation. 

Therefore, by \cite[Lemma 3.2]{Si}, $g(t)=\frac{C_1}{(12+t^2a^2\e^2)^2}+\frac{C_2}{t(12+t^2a^2\e^2)^2}+\frac{C_3}{t^2(12+t^2a^2\e^2)^2}$ is the general solution of the equation (\ref{3rd}). By using this solution  and the variation of parameters, we obtain the general solution of $Sr_{11}$
\begin{equation}\label{Sr11}
\begin{split}
&Sr_{11}(t)=\frac{C_1}{(12+t^2a^2\e^2)^2}+\frac{C_2}{t(12+t^2a^2\e^2)^2}+\frac{C_3}{t^2(12+t^2a^2\e^2)^2}\\
&-\frac{12}{a\e c(1-c)t^2(12+t^2a^2\e^2)^2}\Big(a^3\e^3 t^3(26n-26+6cn-5c-c^2)\\
&+18\sqrt 3(n-1)((9+4c)t^2a^2\e^2-96c-132)\arctan\left(\frac{ta\e\sqrt 3}{6}\right)\\
&-216ta\e(n-1)(5+3c)\log(12+a^2\e^2t^2)\Big). 
\end{split}
\end{equation}
By using this solution and (\ref{Shat}), we obtain $Sr$. 
\begin{equation}\label{Sr12}
\begin{split}
&Sr_{12}(t)=\frac{1}{12 a\e c t^2(1-c)(12+t^2a^2\e^2)^2}\Big(-864\sqrt 3 a\e(n-1)t\\
&(4t^2a^2\e^2c+9t^2a^2\e^2-168c-252)\arctan\left(\frac{ta\e\sqrt 3}{6}\right)\\
&+2592(3c+5)(n-1)(-12+5t^2a^2\e^2)\log(12+t^2a^2\e^2)\\
&-144(n-1)((7+c)t^4a^4\e^4-(576+276c)t^2a^2\e^2-216-432c\\
&+c(1-c)(72t^4a^4\e^4-4C_1t^3a^2\e^2+(144 -5C_2)t^2a^2\e^2\\
&+(24C_1-6C_3a^2\e^2)t+1728+12C_2)\Big)
\end{split}
\end{equation}
and 
\begin{equation}\label{Sr22}
\begin{split}
&Sr_{22}(t)=-\frac{1}{36a^2\e^2t^2c(1-c)(12+t^2a^2\e^2)^2}\Big(-864\sqrt 3a\e(n-1)(324\\
&+144c-405t^2a^2\e^2-264a^2\e^2ct^2+9a^4\e^4t^4+4a^4\e^4t^4c)\arctan\left(\frac{ta\e\sqrt 3}{6}\right)\\
&+15552(n-1)(5+3c)a^2\e^2t(-6+t^2a^2\e^2)\log(12+t^2a^2\e^2)\\
&+288a^2\e^2(n-1)t(3t^4a^4\e^4+ct^4a^4\e^4-102t^2a^2\e^2c-1296-1044c)\\
&-c(1-c)(144C_1+4C_1t^4a^4\e^4-48C_1a^2\e^2t^2-36C_2a^2\e^2t-1296t^3a^4\e^4\\
&+6C_2a^4\e^4t^3-10368ta^2\e^2+9C_3a^4\e^4t^2-144t^5a^6\e^6\Big)-54432nt^3a^4\e^4. 
\end{split}
\end{equation}

By choosing $C_1=0$, $C_3=0$, and,  
\[
C_2=-144+\frac{1296(n-1)(11+8c+(10+6c)\log(12))}{c(1-c)}, 
\]
we obtain the following from (\ref{Sr11}), (\ref{Sr12}), and (\ref{Sr22}), 
\[
\begin{split}
&\left(\begin{array}{cc}
Sr_{11}(t) & Sr_{12}(t)\\
Sr_{12}(t) & Sr_{22}(t)
\end{array}\right)\\
&=\frac{1}{t}\left(\begin{array}{cc}
-1 & 0\\
0 & -\frac{2n-1-c}{1-c}
\end{array}\right)+\left(\begin{array}{cc}
0 & -\frac{a\e(n-2+c)}{2(1-c)}\\
-\frac{a\e(n-2+c)}{2(1-c)} & 0
\end{array}\right)\\
&+t\left(\begin{array}{cc}
-\frac{a^2\e^2(2n-2+c^2-c)}{12c(1-c)} & 0\\
0 & \frac{a^2\e^2(2n-5+3c)}{1-c}
\end{array}\right)+o(t^2). 
\end{split}
\]
It follows from this and Lemma \ref{asye} that 
\begin{equation}\label{ineq}
\hat S(1-t)\geq Sr(t)
\end{equation}
for $t>0$ small enough. It follows from \cite[Proposition 1]{Ro}, (\ref{Shat}), and (\ref{Shat0}) that the above inequality holds for all $t$ in $[0,1]$. 

On the other hand, by (\ref{Sr11}) and (\ref{Sr22}), $\tr(Sr(t))\geq -\frac{f(t a\e)}{t}$, where $f$ is bounded.  Therefore, the result follows from this, (\ref{ineq}), Lemma \ref{masse}, (\ref{S22e}), and (\ref{vole}). 
\end{proof}

\smallskip

\section{A Bonnet-Myers type theorem}

In this section, we show that our approach can be used to prove the following Bonnet-Myers type theorem. 
\begin{thm}\label{Myers}
Assume that the Sasakain manifold has Tanaka-Webster Ricci curvature bounded below by a positive constant $k^2$. Then its diameter with respect to the sub-Riemannian metric is bounded above by a constant depending only on $k$. 
\end{thm}

\begin{proof}
Unless otherwise stated, we use the same notation as in the proof of Theorem \ref{main}. 
Under the assumptions of the theorem, we can suppose that the Tanaka-Webster curvature satisfies $\tr\bar R(t)\geq (2n-1)r^2a^2$. Under this assumption, the equation (\ref{SCC}) becomes 
\[
\begin{split}
&\frac{d}{dt}\bar S(t)\leq \bar S(t)(C_0K+C_4^T) +(KC_0+C_4)\bar S(t)\\
&+KC_0K-s_3C_0K-s_3KC_0-\frac{1}{2n-2}s_3^2C_0-(2n-1)r^2a^2C_0\\
&\leq \bar S(t)D_1(t) +D_1(t)^T\bar S(t)+D_2(t)-(2n-1)r^2a^2C_0
\end{split}
\]

However, $K(t)$, $D_1(t)$, and $D_2(t)$ are given by the followings in this case: 
\[
\begin{split}
&r_t=r\,a\,(1-t), \\
&K=\left(\begin{array}{cc}
K_{11}(t) & K_{12}(t)\\
K_{12}(t) & K_{22}(t)
\end{array}\right)\\
&=\frac{1}{2-r_t\sin(r_t)-2\cos(r_t)}\left(\begin{array}{cc}
r^3\,a\sin(r_t) & r^2\,a(\cos(r_t)-1)\\
r^2\,a(\cos(r_t)-1) & r\,a(\sin(r_t)-r_t\cos(r_t))
\end{array}\right),\\
&D_1(t)=\left(\begin{array}{cc}
0 & a\\
K_{12}(t) & K_{22}(t)
\end{array}\right),\\
&D_2(t)=\left(\begin{array}{cc}
\frac{(2n-2+c)K_{12}(t)^2}{c} & K_{12}(t)K_{22}(t)\\
K_{12}(t)K_{22}(t) & \frac{(2n-c-1)K_{22}(t)^2}{1-c}
\end{array}\right).
\end{split}
\]
Note that since $t\mapsto\varphi_t$ is a length minimizing path connecting the endpoints, $\bar S(t)$ is well-defined for all $t$ in the open interval $(0,1)$. 

Let $Sr(t)=\left(\begin{array}{cc}
Sr_{11}(t) & Sr_{12}(t)\\
Sr_{12}(t) & Sr_{22}(t)
\end{array}\right)$ be a one parameter family of $2\times 2$ matrices which is a solution of the equation 
\[
\begin{split}
&\frac{d}{dt}Sr(t)=-Sr(t)D_1(1-t) -D_1(1-t)^TSr(t) -D_2(1-t) +r^2a^2C_0. 
\end{split}
\]

The function $Sr_{11}(t)$ satisfies the following equation 
\[
\begin{split}
&\frac{d^3}{dt^3}Sr_{11}(t)+\frac{3r a\, r_{1-t}(1-\cos(r_{1-t}))}{2-r_{1-t}\sin(r_{1-t})-2\cos(r_{1-t})}\frac{d^2}{dt^2}Sr_{11}(t)\\
&+\frac{r^2a^2(10-10\cos(r_{1-t})+r_{1-t}^2-4r_{1-t}\sin(r_{1-t})-2r_{1-t}^2\cos(r_{1-t}))}{4-4\cos(r_{1-t})-4r_{1-t}\sin(r_{1-t})+r_{1-t}^2+r_{1-t}^2\cos(r_{1-t})}\frac{d}{dt}Sr_{11}(t)\\
&+\frac{2r^3a^3(r_{1-t}+\sin(r_{1-t})-2r_{1-t}\cos(r_{1-t}))}{4-4\cos(r_{1-t})-4r_{1-t}\sin(r_{1-t})+r_{1-t}^2+r_{1-t}^2\cos(r_{1-t})}Sr(t)=b(t), 
\end{split}
\]
where 
\[
\begin{split}
&b(t)=\frac{a^4r^6(1-\cos(r_{1-t}))^2}{(2-2\cos(r_{1-t})-r_{1-t}\sin(r_{1-t}))^4}\Big(\frac{2(2n-1-c)}{1-c}(\sin(r_{1-t})-r_{1-t}\cos(r_{1-t}))^2\\
&+\frac{2n-2+c}{c}(2(r_{1-t}\cos(r_{1-t})-\sin(r_{1-t}))^2+(2-r_{1-t}^2+2\cos(r_{1-t}))(1-\cos(r_{1-t})))\\
&-2(2(r_{1-t}\cos(r_{1-t})-\sin(r_{1-t}))^2-(r_{1-t}-\sin(r_{1-t}))^2)\\
&-2(2n-1)(1-\cos(r_{1-t}))(4-4\cos(r_{1-t})-4r_{1-t}\sin(r_{1-t})+r_{1-t}^2+r_{1-t}^2\cos(r_{1-t}))\Big). 
\end{split}
\]

The corresponding homogeneous equation is given by 
\[
\begin{split}
&\frac{d^3}{dt^3}g(t)+\frac{3r a\, r_{1-t}(1-\cos(r_{1-t}))}{2-r_{1-t}\sin(r_{1-t})-2\cos(r_{1-t})}\frac{d^2}{dt^2}g(t)\\
&+\frac{r^2a^2(10-10\cos(r_{1-t})+r_{1-t}^2-4r_{1-t}\sin(r_{1-t})-2r_{1-t}^2\cos(r_{1-t}))}{4-4\cos(r_{1-t})-4r_{1-t}\sin(r_{1-t})+r_{1-t}^2+r_{1-t}^2\cos(r_{1-t})}\frac{d}{dt}g(t)\\
&+\frac{2r^3a^3(r_{1-t}+\sin(r_{1-t})-2r_{1-t}\cos(r_{1-t}))}{4-4\cos(r_{1-t})-4r_{1-t}\sin(r_{1-t})+r_{1-t}^2+r_{1-t}^2\cos(r_{1-t})}g(t)=0
\end{split}
\]
which is the second symmetric power of the equation 
\[
\begin{split}
&\frac{d^2}{dt^2} f(t)+\frac{rar_{1-t}(1-\cos(r_{1-t}))}{2-2\cos(r_{1-t})-r_{1-t}\sin(r_{1-t})}\frac{d}{dt} f(t)\\
&+\frac{r^2a^2(1-\cos(r_{1-t}))}{2-2\cos(r_{1-t})-r_{1-t}\sin(r_{1-t})}=0. 
\end{split}
\]

It follows that the function $h=\left(r_{1-t}\cos\left(\frac{1}{2}r_{1-t}\right)-2\sin\left(\frac{1}{2}r_{1-t}\right)\right)f$ satisfies 
\[
\frac{d^2}{dt^2} h(t)+\frac{1}{4}r^2a^2h(t)=0
\]
and the general solution is given by 
\[
h(t)=c_1\cos\left(\frac{1}{2}r_{1-t}\right)+c_2\sin\left(\frac{1}{2}r_{1-t}\right). 
\]
Therefore, the general solution of the original third order homogeneous equation is given by 
\[
g(t)=\frac{C_1\cos^2\left(\frac{1}{2}r_{1-t}\right)+C_2\cos\left(\frac{1}{2}r_{1-t}\right)\sin\left(\frac{1}{2}r_{1-t}\right)+C_3\sin^2\left(\frac{1}{2}r_{1-t}\right)}{m(t)}, 
\]
where $m(t)=\left(\frac{1}{2}r_{1-t}\cos\left(\frac{1}{2}r_{1-t}\right)-\sin\left(\frac{1}{2}r_{1-t}\right)\right)^2$. 

By the variation of parameters, a particular solution to the original inhomogeneous equation is 
\[
\begin{split}
&Sr_{11}(t)=\frac{2\cos^2\left(\frac{1}{2}r_{1-t}\right)}{r^2a^2m(t)}\int_0^tb(s)m(s)\sin^2\left(\frac{1}{2}r_{1-s}\right)ds\\
&-\frac{4\sin\left(\frac{1}{2}r_{1-t}\right)\cos\left(\frac{1}{2}r_{1-t}\right)}{r^2a^2m(t)}\int_0^tb(s)m(s)\sin\left(\frac{1}{2}r_{1-s}\right)\cos\left(\frac{1}{2}r_{1-s}\right)ds\\
&+\frac{2\sin^2\left(\frac{1}{2}r_{1-t}\right)}{r^2a^2m(t)}\int_0^tb(s)m(s)\cos^2\left(\frac{1}{2}r_{1-s}\right)ds. 
\end{split}
\]

A computation shows that 
\[
t^3Sr_{11}(t)=\frac{-\frac{192(2n-c-1)}{1-c}-\frac{156(2n-2+c)}{c}+336}{a^2}+o(t)
\]
at $t\to 0$. 

By using the equation satisfied by $Sr(t)$, we can compute and obtain 
\[
t^2Sr_{12}(t)=\frac{6(\frac{8(2n-c-1)}{1-c}+\frac{7(2n-2+c)}{c}-14)}{a}+o(t)
\]
and 
\[
tSr_{22}(t)=-4\left(\frac{4(2n-c-1)}{1-c}+\frac{3(2n-2+c)}{c}-6\right). 
\]
as $t\to 0$. 

It follows as in the proof of Theorem \ref{main} that $\bar S(1-t)\geq Sr(t)$. On the other hand, let $t_0>0$ be the first positive number such that $m(t_0)=0$. It follows that 
\[
\begin{split}
&Sr_{11}(t)=\frac{2\cos^2\left(\frac{1}{2}r_{1-t}\right)}{r^2a^2m(t)}\int_0^tb(s)m(s)\Big[\sin\left(\frac{1}{2}r_{1-s}\right)-\frac{1}{2}r_{1-t}\cos\left(\frac{1}{2}r_{1-s}\right)\Big]^2 ds. 
\end{split}
\]
Therefore, $Sr_{11}(t)$ goes to $\infty$ as $t\to t_0$. This is a contradiction if $t_0<1$. Hence, $a\leq \frac{2x_0}{r}$ where $x_0$ is the first positive root of the function $x\mapsto x\cos(x)-\sin(x)$. 
\end{proof}

\end{document}